\documentclass{amsart}
\usepackage{comment}
\usepackage{amsmath, amssymb}
\usepackage{amsthm}
\usepackage{graphicx}
\usepackage{hyperref}

\DeclareMathOperator{\Vol}{Vol}
\DeclareMathOperator{\inorm}{in}

\title{Determinant of the finite volume {L}aplacian}
\author{Thomas Doehrman\and
David Glickenstein}
\thanks{TD and DG partially supported by NSF DMS 1760538. DG partially supported by NSF CCF 1740858 and NSF DMS 1937229}
\begin{document}

\begin{abstract}
    The finite volume Laplacian can be defined in all dimensions and is a natural way to approximate the operator on a simplicial mesh. In the most general setting, its definition with orthogonal duals may require that not all volumes are positive; an example is the case corresponding to two-dimensional finite elements on a non-Delaunay triangulation. Nonetheless, in many cases two- and three-dimensional Laplacians can be shown to be negative semidefinite with a kernel consisting of constants. This work generalizes work in two dimensions that gives a geometric description of the Laplacian determinant; in particular, it relates the Laplacian determinant on a simplex in any dimension to certain volume quantities derived from the simplex geometry. 
\end{abstract}

\maketitle

\newcommand{\R}{\mathbb{R}}
\newcommand{\norm}[1]{\left\lVert#1\right\rVert}

\newtheorem{theorem}{Theorem}
\newtheorem{acknowledgement}[theorem]{Acknowledgement}
\newtheorem{algorithm}[theorem]{Algorithm}
\newtheorem{axiom}[theorem]{Axiom}
\newtheorem{case}[theorem]{Case}
\newtheorem{claim}[theorem]{Claim}
\newtheorem{conclusion}[theorem]{Conclusion}
\newtheorem{condition}[theorem]{Condition}
\newtheorem{conjecture}[theorem]{Conjecture}
\newtheorem{corollary}[theorem]{Corollary}
\newtheorem{criterion}[theorem]{Criterion}
\newtheorem{definition}[theorem]{Definition}
\newtheorem{example}[theorem]{Example}
\newtheorem{exercise}[theorem]{Exercise}
\newtheorem{lemma}[theorem]{Lemma}
\newtheorem{notation}[theorem]{Notation}
\newtheorem{problem}[theorem]{Problem}
\newtheorem{proposition}[theorem]{Proposition}
\newtheorem{remark}[theorem]{Remark}
\newtheorem{solution}[theorem]{Solution}
\newtheorem{summary}[theorem]{Summary}

\section{Introduction}
The finite volume Laplacian is an important object in the study of numerical partial differential equations. It arises, for instance, by considering the approximation of functions on a triangulation by piecewise constant functions with support on control volumes associated to the vertices (see, e.g., \cite{lichenwu}). This type of Laplacian also arises naturally in variation formulas of discrete curvatures with respect to circle packing and other types of discrete conformal structures, both in two and three dimensions, e.g.,  \cite{he, chowluo, glickenstein1, glickenstein4}. In two dimensions, finite element Laplacians can be considered a special case of finite volume Laplacians, where the control volumes arise from circumcentric dual vertices \cite{bankrose}.

It is notable that finite volume Laplacians make sense even when control volumes are negative or have negative contributions, as we see in Section \ref{sec:lap}. These arise naturally as variations of discrete conformal structures in dimensions two and three, and also arise naturally from two-dimensional finite elements on non-Delaunay triangulations. In many such cases, the Laplacian is still of maximal rank with nonpositive eigenvalues, even though the Laplacian is no longer a graph Laplacian with positive weights because some edge weights are negative.

In \cite[Section 5]{glickenstein3}, it is shown how to relate the Laplacian determinant of a single triangle to the ratio of the area of the pedal triangle of the triangle center to the area of the triangle. This demonstrates that many finite volume Laplacians have full rank, regardless of the positivity of edge weights. This article generalizes the formula for the Laplacian determinant to all dimensions. Some special cases of these were studied in other ways in the works \cite{glickenstein2, hexu}. 

\section{Preliminaries}
\subsection{Duality structures on simplices and triangulation}\label{sec:duals}
In order to properly define finite volume Laplacians, we need to have a structure to separate space into pieces related to duals of the simplex. Often it is assumed that the finite volume decomposition produces subdivisions of each simplex with all positive volumes, often called ``well-centered'' (e.g., \cite{hirani2003}) or that the dual volumes all have nonnegative generalized areas, a property that in two dimensions is equivalent to the Delaunay or weighted Delaunay condition \cite{glickenstein2005geometric}. In this work, we choose not to make an assumption of nonnegative generalized volumes on dual cells. For instance, there is a finite volume Laplacian corresponding to a two-dimensional finite element structure (see, e.g., \cite{bankrose}) from a non-Delaunay triangulation will have some pieces of the decomposition without positive generalized volume.

In order to make sense of positive/negative generalized volumes, we give a formal definition of duality structures. These ideas are well developed in \cite{hirani2003,desbrunetal2005, glickenstein2005geometric,Schneiderpolytopes2020}. We begin with the combinatorial description of simplices as ordered numbers such as $[0,1,2]$ as described in any algebraic topology text (e.g., \cite{hatcher}). Choosing vertices $v_0, v_1, \dots v_n$ in some Euclidean space $\R^n$ gives a geometry to that simplex, with lengths of edges and volumes of subsimplices. 
Dual cells are constructed from the Poincar\'e duals.
The geometry of dual cells can be constructed for a single simplex by projecting a chosen center onto each face. 
Often we will refer only to length, area, and volume rather than ``generalized'' length, area, and, volume. We will sometimes call the generalized volume of dual cells ``dual volumes.''

\subsubsection{Dual lengths in a triangle} \label{sec:duals2D}
We begin in two dimensions; refer to Figure \ref{fig:dualitytriangle}. Given a Euclidean triangle embedded in $\R^2$ with vertices $\{ v_0, v_1, v_2\}$ and any point $C_{[0,1,2]} \in \R^2$ (called the ``center'' of the face $[0,1,2]$), one can define a center for each edge by orthogonal projection onto the lines through edges. The center of edge $[i,j]$,  denoted $C_{[i,j]}$, is given for any triangle $[i,j,k]$ by:
\[
    C_{[i,j]} = \frac{(v_j - v_i) \cdot (C_{[i,j,k]} - v_i)}{\norm{v_j - v_i}^2} (v_j - v_i) + v_i.
\]

From here we can define the dual length $\ell_{ij,k}^*$ associated with edge $[i,j]$ in triangle $[i,j,k]$ as the signed distance between $C_{[i,j,k]}$ and $C_{[i,j]}$
\[
    \ell_{ij,k}^* = \pm \norm{C_{[i,j]} - C_{[i,j,k]}} 
\]
where $\ell_{ij,k}^* > 0$ when $C_{[i,j,k]}$ lies on the same side of $[i,j]$ as the realization of the triangle $[i,j,k]$. Another way to define this with a more clear description of the sign is to define the inward normal 
\[ \inorm_{ij} = (v_k - v_i)-\frac{(v_k - v_i)\cdot (v_j - v_i)}{\norm{v_j - v_i}^2}(v_j-v_i)
\]
and then 
\[
    \ell_{ij,k}^* = (C_{[i,j,k]}-C_{[i,j]})\cdot \inorm_{ij}/\norm{\inorm_{ij}}.
\]

\begin{figure}[ht]
    \centering
    \includegraphics[width=.49\textwidth]{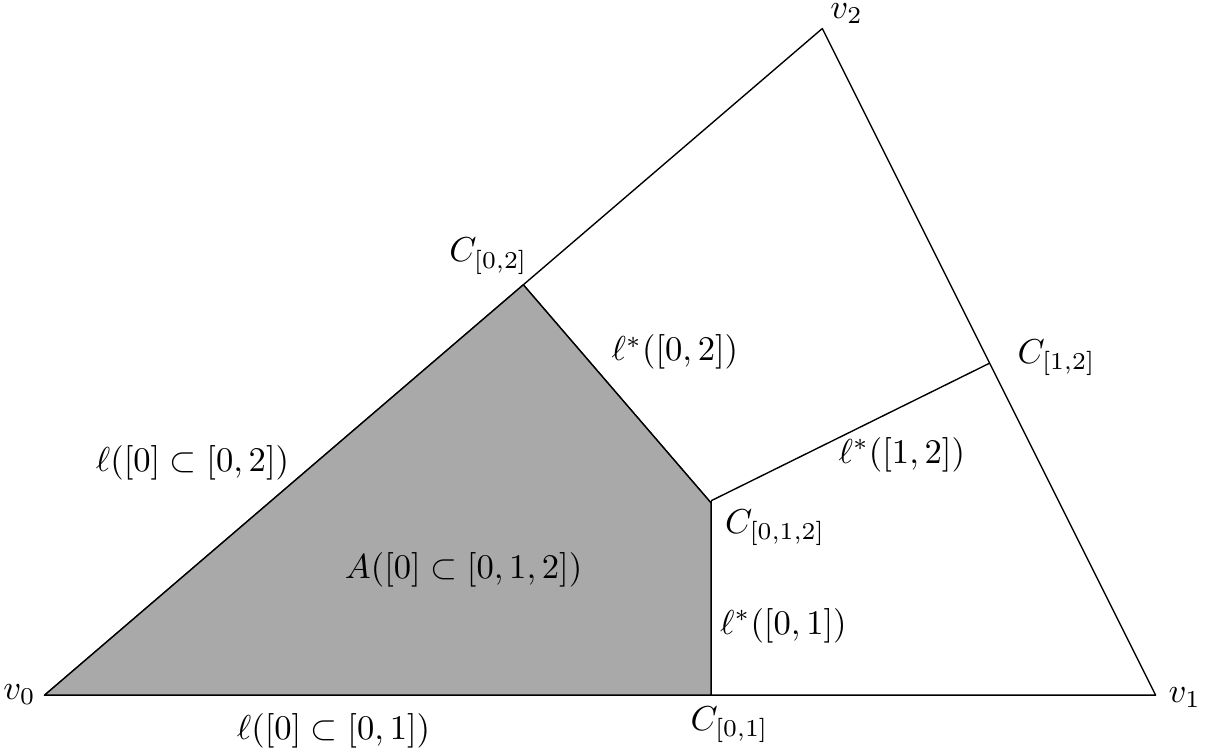}
    \includegraphics[width=.49\textwidth]{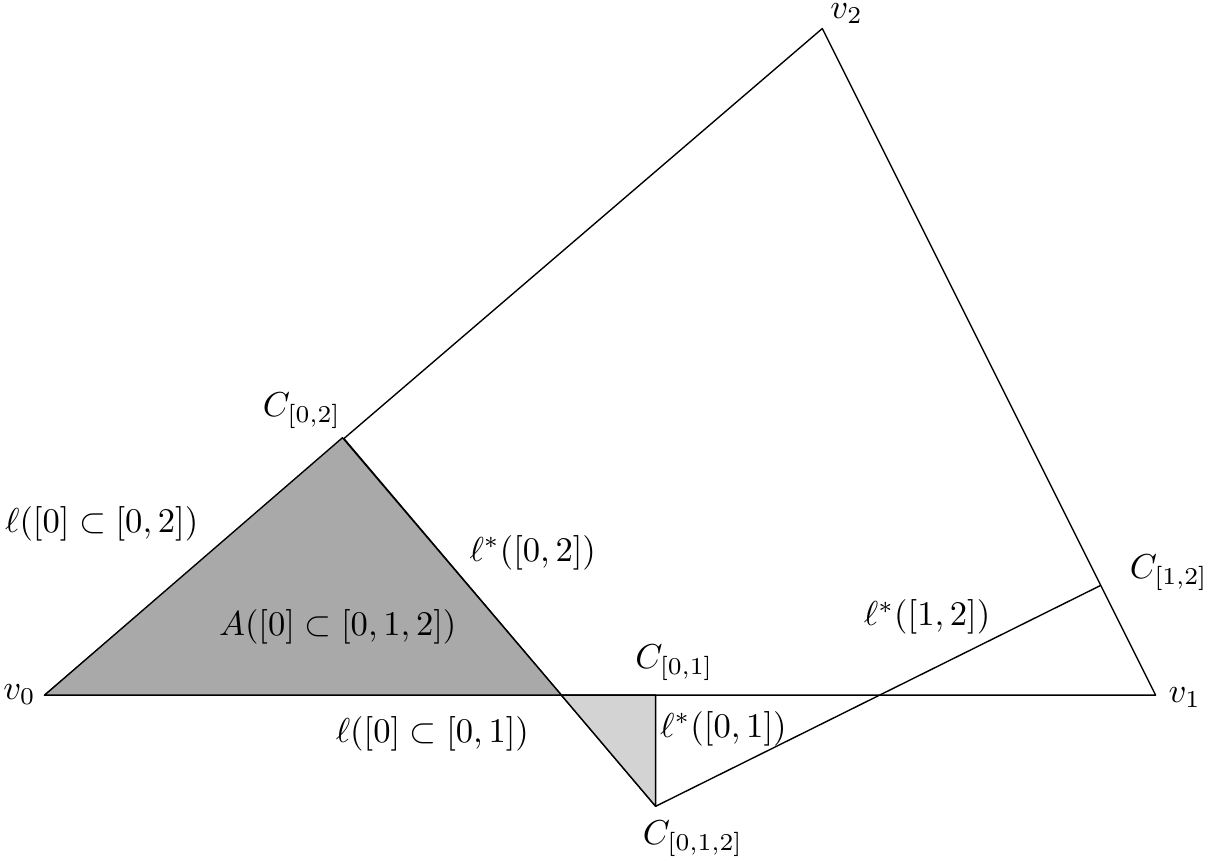}
    \caption{Duality structures on triangles $[0,1,2]$ with center $C_{[0,1,2]}$ inside (left) and outside (right) the triangle.}
    \label{fig:dualitytriangle}
\end{figure}

Note that the centers also divide up the primal simplices. In particular, each edge $[i,j]$ is divided into 
one piece joining $v_i$ and $C_{[i,j]}$  and another piece joining $v_j$ and $C_{[ij]}$, resulting in signed lengths $\ell([i] \subset [i,j])$ and $\ell([j] \subset [i,j])$ with the property that 
\[
    \ell([i] \subset [i,j]) + \ell([j] \subset [i,j]) = \ell([i,j]).
\]
These can be seen in Figure \ref{fig:dualitytriangle}. In \cite{glickenstein2005geometric} and other of the second author's work, typically $\ell([i] \subset [i,j])$ is denoted as $d_{ij}$. Note that in the notation $\ell(\tau \subset \sigma)$ and similar quantities, the value depends on both $\tau$ and $\sigma$ but is written in this way to emphasize which simplex is a subsimplex of the other.

This structure provides a way to divide other pieces of simplices into parts. Notably, we have $A([i] \subset [i,j,k])$, which is the piece of area associated to vertex $i$, as seen in Figure \ref{fig:dualitytriangle}. Note that in the right picture, the heavy gray piece is considered positive area and the light gray piece is considered negative area, with the two signed contributions added together to determine $A([0]\subset [0,1,2])$.

While dual volumes can also be defined for faces and vertices, the finite volume Laplacian is related specifically to the dual volumes of codimension one (i.e. volumes associated with edges in the original simplex). It will be helpful to illustrate how this process works in three dimensions before generalizing to dimension $N$.

\subsubsection{Dual areas in a tetrahedron}
Refer to Figure \ref{fig:dualitytetra}. Let $\{v_0, v_1, v_2, v_3\}$ be the vertices of a Euclidean tetrahedron, and pick a center $C_{[0,1,2,3]} \in \R^3$. For any simplex $\sigma \subset [0,1,2,3]$, let $C_\sigma$ denote the orthogonal projection of $C_{[0,1,2,3]}$ onto the line or plane containing $\sigma$. For two nested simplices $\tau \subset \sigma$, we define dual lengths $\ell^*({ \tau \subset \sigma})$ by

\[
    \ell^*({ \tau \subset \sigma}) = \pm \norm{C_{\sigma} - C_{\tau}}
\]
where $\ell^*( \tau \subset \sigma) > 0$ when $C_\sigma$ lies on the same side of $\tau$ as $\sigma$.

Then the dual area associated with the edge $[i,j]$ in $[i,j,k]$ is given by
\[
    A^*([i,j]) = \sum_{k,\ell:[i,j] \subset [i,j,k] } \frac{\ell^*\left({ [i,j] \subset [i,j,k]}\right)\ell^*\left({ [i,j,k] \subset [i,j,k,l]}\right)}{2}.
\]
Note that $A^*([i,j])$ is the sum of the signed areas of two right triangles with vertices $\{C_{[i,j]}, C_{[i,j,k]}, C_{[i,j,k,l]}\}$ and $\{C_{[i,j]}, C_{[i,j,l]}, C_{[i,j,k,l]}\}$. We are now ready to generalize to $N$-dimensions.

\begin{figure}[ht] 
    \centering
    \includegraphics[width=.95\textwidth]{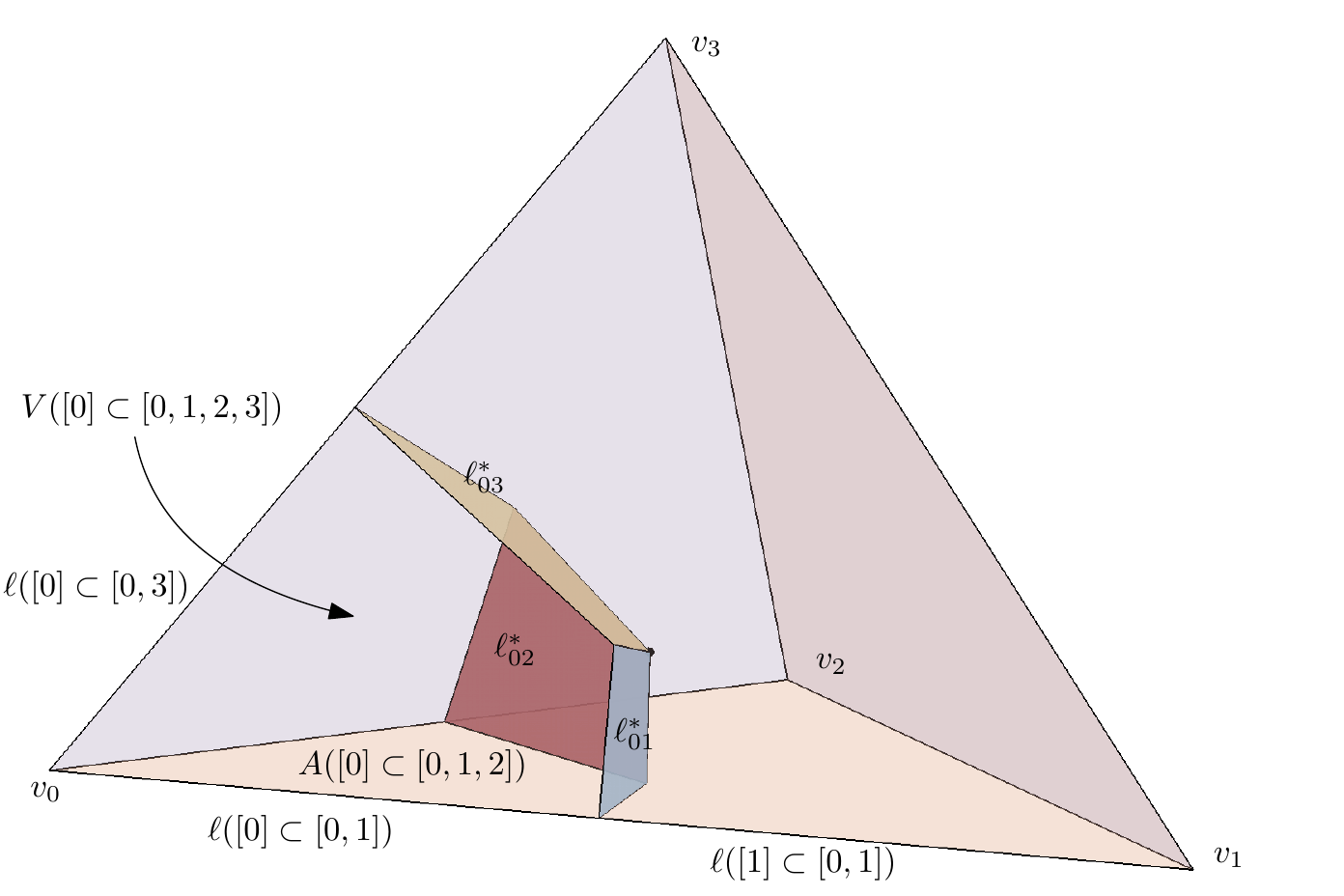}
    \caption{View of the dual structure in a tetrahedron. Note that $\ell^*_{ij}=\ell^*([i,j]\subset [0,1,2,3])$ as described in Remark \ref{rmk:lij}.}
    \label{fig:dualitytetra}
\end{figure}

\subsubsection{Dual volumes in an \texorpdfstring{$N$}{TEXT}-simplex} \label{sec:dualnd}

Let $\{v_k\}_{k = 0}^N \subset \R^N$ be the vertices of a Euclidean $N$-simplex $T$, and pick a center $C_{T} \in \R^N$. For any simplex $\sigma \subset T$ with dimension at least one, let $C_\sigma$ denote the orthogonal projection of $C_{T}$ onto the hyperplane containing $\sigma$ of the same dimension as $\sigma$. For two nested simplices $\tau \subset \sigma$, we define dual lengths $\ell^*({ \tau \subset \sigma})$ just as in three dimensions:
\[
    \ell^*({ \tau \subset \sigma}) = \pm \norm{C_{\sigma} - C_{\tau}}
\]
where $\ell^*( \tau \subset \sigma) > 0$ when $C_\sigma$ lies on the same side of $\tau$ as $\sigma$. In Section \ref{sec:duals2D} we used the notation $\ell_{ij,k}^*$ more succinctly to represent $\ell^*([i,j] \subset [i,j,k])$. 

The dual volume associated with the edge $[i,j]$ is given by
\[
    V^*([i,j]\subset T) = \frac{1}{(N-1)!}\sum_{[i,j] = \sigma_1 
    \subset \cdots \subset \sigma_N = T } \prod_{n = 1}^{N-1}\ell^*\left({ \sigma_n \subset \sigma_{n+1}}\right)
.\]
$V^*([i,j]\subset T)$ is the sum of the signed volumes of $(N-1)!$ orthoschemes of dimension $N-1$, each with vertices $\{ C_{\sigma_1}, C_{\sigma_2}, \ldots , C_{\sigma_N} \}$ for some nested collection of simplices $[i,j] = \sigma_1 \subset \sigma_2 \subset \cdots \subset \sigma_N = T $. For a description of how signed volumes of orthoschemes add up to the volume of a simplex, see \cite{Schneiderpolytopes2020}.

We will sometimes use $A$ to denote $2$-dimensional volumes or to denote an $N-1$ dimensional volume in a $N$-dimensional simplex to emphasize the difference between the two types of volumes. In particular, for a simplex $\sigma = [0,1,\ldots,N]$ we may write $A_{\widehat{i}}$ to denote the $(N-1)$-dimensional volume of simplex $\sigma_{\widehat{i}} = [0,1,\ldots,\widehat{i}, \ldots, N]$, where the hat denotes the vertex is missing. A quantity that will be important in the sequel is the area associated to a vertex, 
\begin{equation}
    A_{j,\widehat{i}} = A([j] \subset \sigma_{\widehat{i}}) = \frac1N\sum_{k\neq i,j}\ell([j] \subset [j,k]) \ell^*([j,k]\subset \sigma_{\widehat{i}})
    \label{eq:Aji}
\end{equation}
if $i\neq j$ and $A_{i,\widehat{i}} = 0$ for all $i$. In dimension 3, we may use $A_{i,jk}$ to denote $A_{i,\widehat{\ell}}$ in a simplex $[i,j,k,\ell]$.

\begin{remark}\label{rmk:lij}
In the rest of the paper, we will mostly be interested in dual areas of the form $\ell^*([i,j] \subset \sigma)$ where $\sigma$ is a simplex containing $[i,j]$. In this case, we will often write $\ell^*([i,j])$ or $\ell^*_{ij}$ instead for brevity, despite the possibility of confusion when we have a triangulation as described in the next section.
\end{remark}


\subsubsection{Remarks on triangulations}

In order to produce a duality structure on a triangulation from duality structures on individual simplices, it is necessary to choose centers in such a way that dual lengths agree on intersections of neighboring simplices. Given two $N$-simplices $T$ and $S$ in a triangulation whose intersection contains an edge $e$, consider a local Euclidean embedding of the triangulation whose domain contains both $T$ and $S$. Within that embedding, let $C(e \subset T)$ and $C(e \subset S)$ denote the centers assigned to the edge $e$ by the simplices $T$ and $S$ respectively.
If $C(e \subset T) = C(e \subset S)$ for every such $T$, $S$, and $e$, then it can be shown that faces of higher dimension will also be assigned consistent centers \cite{glickenstein2005geometric}. This is equivalent to expressing the fact that the lengths and dual structures are entirely determined by the quantities $\ell([i] \subset [i,j])$, and this parametrization is explored specifically in \cite{glickenstein2005geometric, glickensteinthomas}. That is, for any simplices $\sigma \subset T \cap S$, we have that $C(\sigma \subset T) = C(\sigma \subset S)$. This condition ensures that the dual volumes described in Section \ref{sec:dualnd} may be pieced together to form local geometric realizations of the cells of the Poincar\'e dual to the original triangulation. Further, for any face $\sigma$ in the original triangulation, the linear space spanned by the realization of $\sigma$ is orthogonal to the linear space spanned by the realization of its Poincar\'e dual. The volume of the dual to a simplex can be defined in general as 
\[
    V^*(\sigma) = \sum_{\sigma^N} V^*(\sigma \subset \sigma^N),
\]
where the sum is over all top dimensional simplices.

\subsection{Graph Laplacians}\label{sec:lap}
We define the weighted graph Laplacian on a graph $G$ with weights $w_{ij}$ as the operator on functions $f:V \to \R$ given by
\[
    L f_i = \sum_{ij\in E} w_{ij} (f_j - f_i).
\]
The background for weighted graph Laplacians can be found in the books by Chung \cite{chung} and Bollob\'{a}s \cite[Chapter II.3]{bollobas}. When the weights are positive, the Laplacian is negative semidefinite with zero eigenvalue precisely on the space spanned by $f_i=1$ for all vertices $i$ on a connected component. This follows because the matrix $L$ is diagonally dominant. 

In this paper, we will not require the weights to be nonnegative, but we will consider Laplacians that are determined by dual structures such that 
\[
    w_{ij} = \frac{\ell^*_{ij}}{\ell_{ij}}.
\]
For this reason, we do not automatically know that the Laplacian has maximal rank or nonpositive eigenvalues, as we do for Laplacian matrices with positive weights. To study this, we will consider the Laplacian determinant on each simplex. The Laplacian determinant, or Kirchoff determinant, can be defined in general on a weighted graph as follows.
\begin{definition} \label{def:K}
The \emph{Laplacian determinant} is 
\begin{equation}\label{K*}
K^*(G)=\det (-\widehat{L}_{00}),
\end{equation}
where $\widehat{L}_{00}$ denotes the Laplacian matrix with the first row and first column removed. 
\end{definition}
By considering the adjugate of the Laplacian matrix, it is easy to see that $K^*(G)=(-1)^{i+j} \widehat{L}_{ij}$ for any choice of $i$th row and $j$th column removed, and that it is the product of all eigenvalues except the first eigenvalue, which is zero. For more, see \cite{chung,bollobas}. Additionally, a generalization of the Matrix Tree Theorem shows that the value can be computed by summing over all spanning trees the product of the weights in that tree (see \cite{duvaletal2009} for this and generalizations).

\subsection{Volumes of simplices}
We will also need some basics for volumes of polyhedra. The usual way to calculate signed volumes of simplices is by considering the vectors emanating from a vertex (see, e.g., \cite{stein66}), i.e., 
\begin{align}
 \Vol(\sigma) &= \frac{(-1)^N}{N!} \det \left[ \begin{array}[c]{ccccc} 
    v_1-v_0 & v_2-v_0 & \cdots & v_N-v_0 \end{array} 
    \right] \label{eq:volume1}\\ 
    &= \frac{(-1)^N}{N!} \det \left[ \begin{array}[c]{ccccc} 
    v_0 & v_1 & \cdots & v_N\\
    1 & 1& \cdots & 1\end{array} 
    \right] \label{eq:volume2},
\end{align}
if $\sigma$ is the simplex determined by the vertices $v_0, v_1, \ldots, v_N$ in $\R^N$, considered as column vectors, where the ordering determines the sign of the volume according to rules of determinants.

In addition, we can see that the volume can be determined by the normals and areas of the faces. While it follows from Minkowski's Theorem that a polytope is uniquely determined up to translation by the normals and areas of the faces (see, e.g., \cite{alexandrov,schneiderconvexbodies}), it is much more straightforward in the case of simplices. While we expect that these arguments are well-known, we have not found a reference. The work was inspired by the derivation of the cosine laws for the sphere and hyperbolic space in \cite{Thurston97}. 


\begin{proposition}\label{prop:detnormal}
Let $n_{\widehat{0}}, n_{\widehat{1}}, \ldots, n_{\widehat{N}}$ denote the outward pointing normals to an $N$-simplex $\sigma$ with vertices $v_0,\ldots,v_n$ in $\R^{N}$, where $n_{\widehat{j}}$ is orthogonal to the plane containing all vertices except $v_j$, and such that the length of $n_{\widehat{j}}$ is equal to the area of the face containing all vertices except $v_j$. Then for any $j=0,\ldots,N$,
\begin{equation}
    \det\left[ \begin{array}[c]{cccccc} 
   n_{\widehat{0}} & n_{\widehat{1}} &\cdots& \widehat{n_{\widehat{j}}} &\cdots & n_{\widehat{N}}
   \end{array}
   \right]  =  \frac{(-1)^{N+j}N^{N}}{N!}\Vol(\sigma)^{N-1}\label{eq:volfromnormals}
\end{equation}
where the hat denotes the vector is not present.
\end{proposition}

\begin{proof}
We see that 
\begin{align}
    \left[ \begin{array}[c]{cc} 
    n_{\widehat{0}}^T & N\Vol(\bar{\sigma}_{\widehat{0}} )\\[.5em] n_{\widehat{1}}^T & N\Vol(\bar{\sigma}_{\widehat{1}})\\ \vdots &\vdots\\ n_{\widehat{N}}^T & N\Vol(\bar{\sigma}_{\widehat{N}})\end{array} 
    \right] 
    \left[ \begin{array}[c]{ccccc} 
    v_0 & v_1 & \cdots & v_N\\
    1 & 1& \cdots & 1\end{array} 
    \right] 
    = N \Vol(\sigma) I_N,  
\label{eq:normalvertexinverse}
\end{align}
where $\bar{\sigma}_{\widehat{j}}$ denotes the $N-$simplex with vertices the origin and the $(N-1)-$simplex excluding vertex $j$, and $\Vol$ denotes signed volume. This follows from the fact that $n_{\widehat{j}}^T v_j = -N \Vol(\bar{\sigma}_{\widehat{j}})$ for each $j$. Notice that $\sum_{j=0}^{N} \Vol(\bar{\sigma}_{\widehat{j}}) = \Vol(\sigma)$. Using the fact that $\sum n_{\widehat{i}}=0$, we see that the determinant on the left side of (\ref{eq:volfromnormals}) does not depend on $j$ except for the sign. It follows by expanding in the last column that the determinant of the leftmost matrix in (\ref{eq:normalvertexinverse}) is equal to 
\begin{align}
(-1)^{N+j}\det\left[ \begin{array}[c]{cccccc} 
   n_{\widehat{0}} & n_{\widehat{1}} &\cdots& \widehat{n_{\widehat{j}}} &\cdots & n_{\widehat{N}}
   \end{array}
   \right]N \Vol(\sigma) \label{eq:normaltimesvertex}
\end{align}
for any $j$. The second matrix on the left has determinant equal to $N!\Vol(\sigma)$. Finally, the matrix on the right has determinant equal to $N^{N+1}\Vol(\sigma)^{N+1}$.

Here is an alternative, but similar proof.
Since $n_{\widehat{0}}+n_{\widehat{1}}+\cdots+n_{\widehat{N}}=0$, it is sufficient to prove this for $j=0$. We see that 
\[
    \left[ \begin{array}[c]{c} 
    n_{\widehat{1}}^T \\[.5em] n_{\widehat{2}}^T \\ \vdots \\ n_{\widehat{N}}^T \end{array} 
    \right] 
    \left[ \begin{array}[c]{ccccc} 
    v_1-v_0 & v_2-v_0 & \cdots & v_N-v_0 \end{array} 
    \right] 
    = -N \Vol(\sigma) I_N
\]
where $I_N$ is the $N \times N$ identity matrix.
The determinant of the second matrix on the left is equal to $(-1)^N N!\Vol(\sigma)$ and so taking determinants and solving, we complete the proof.
\end{proof}

\section{Formula for the Laplacian determinant}

In this section we calculate the Laplacian determinant on a simplex, which gives the following theorem.
\begin{theorem} \label{thm:determinant}
Let $T=[0,1,\ldots, N]$ be a $N$-simplex realized with vertices $v_0, \ldots, v_N \in \mathbb{R}^N$ with duality structure determined by $C_T\in \R^N$. The Laplacian determinant $K^*\left(  T\right)$ is defined as in Definition \ref{def:K}. Let the dual simplex $T^\#$ be the simplex determined by the normals $n_{i} = \sum_j n_{ij}$ where 
\begin{align}
   n_{ij} = \frac{\ell^*_{ij}}{\ell_{ij}} (v_j - v_i), \label{def:n}
\end{align}
and let $\widehat{n}$ be the matrix with columns $n_1,\ldots, n_N$. 
Then
\begin{align}
    K^*\left(  T\right) = -\frac{\det (\widehat{n})}{6\Vol(T)}.
\end{align}
Furthermore, we have 
\begin{align}
K^*\left(  T\right) = \frac{N^{N}\Vol(T^{\#})^{N-1}}{\left(N!\right)^2\Vol(T)}. \label{eq:keqVoverV}
\end{align}

Finally, let $A_{\widehat{i}}$ equal the volume of $\sigma_{\widehat{i}}$ and $A_{j,\widehat{i}}$ denote the signed volume of the piece of $\sigma_{\widehat{i}}$ associated to vertex $j$ as in (\ref{eq:Aji}). Let $A$ be the matrix whose $ij$th elements are $A_{j,\widehat{i}}$. Then 
\begin{align}
    K^*\left(  T\right) = \frac{ (-1)^N N^N  \Vol(T)^{N-2}\det(A)}{(N!)^2\prod_{i=0}^N A_{\widehat{i}}}.
\label{eq:KandA}
\end{align}
\end{theorem}

We first consider the case of three dimensions. The two dimensional case is relatively simple and is described in \cite{glickenstein3}, so it is skipped here. We will discuss it more in Section \ref{sec:simson}.

\subsection{Three dimensions}
Refer to Figure \ref{fig:dualitytetra}.
The key observation is that 
\[
    n_{ij}=-n_{ji}=\frac{ \ell^*_{ij}}{\ell_{ij}} (v_j-v_i)
\]
is a vector that is normal to the dual face with length equal to the area of that dual face.
We then define 
\[
    n_i = -(n_{ij}+n_{ik}+n_{i\ell})
\]
for $\{i,j,k,\ell\}=\{0,1,2,3\}$ denoting the vertices of the tetrahedron.

Notice that 
\[
    n_0 + n_1 +n_2 +n_3 = 0
\]
and so it follows from Minkowski's Theorem (see, e.g., \cite[Chapter 7]{alexandrov}) that these four vectors determine the faces of a tetrahedron $T^{\#}$ with areas equal to the lengths of the vectors and with those vectors normal to the faces. Notice the following relationship between the vertices of the tetrahedron, the Laplacian matrix, and the normals.
\[
n_{i}=-\sum\frac{\ell_{ij}^{\ast}}{\ell_{ij}}\left(
v_{j}-v_{i}\right).
\]%
We will express this in matrix form as $-Lv^T=n^T$,
where $v$ has columns $v_0$, $v_1$, $v_2$, and $v_3$ and $n$ has columns $n_0$, $n_1$, $n_2$, and $n_3$.
%
We now calculate $K^*(T)= \det (-\widehat{L}_{00})$.

For any matrix $M$ let $\widehat{M}$ denote the matrix $M$ with the first row removed. We then use the
Binet-Cauchy formula and the fact that $K^*\left(  T\right)  =\left(  -1\right)^{i+j}\det (-\widehat{L}_{ij})$ for any choice of $i$ and $j$ to see that

\[\det\left(  -\widehat{L}v^T\right) = K^*\left(  T\right) \left( \det\left[ \begin{array}[c]{c}
    v_1^T \\ v_2^T \\ v_3^T \end{array} 
    \right]
    - \det\left[ \begin{array}[c]{c}
    v_0^T \\ v_2^T \\ v_3^T \end{array} 
    \right]
    + \det\left[ \begin{array}[c]{c}
    v_0^T \\ v_1^T \\ v_3^T \end{array} 
    \right]
    - \det\left[ \begin{array}[c]{c}
    v_0^T \\ v_1^T \\ v_2^T \end{array} 
    \right] \right)
.\]
To compute the alternating sum of determinants on the right hand side, we  add and subtract $v_0^T$ from each row of the first term and then expand.
Thus, we have
\[ \det\left( - \widehat{L}v^T\right) = K^*\left(  T\right) \left( \det\left[ \begin{array}[c]{c}
    v_1^T - v_0^T \\ v_2^T - v_0^T \\ v_3^T - v_0^T \end{array} 
    \right]
    \right) =-6K^*\left(  T\right)  \Vol\left(  T\right)
\]
from (\ref{eq:volume1}) and
\[
\det\left(\widehat{n^T}\right)=-\frac{9}{2}\Vol\left(  T^{\#}\right)  ^{2}%
\]
by (\ref{eq:volfromnormals}) in Proposition \ref{prop:detnormal}. It follows from $-\widehat{L}v^T = \widehat{n^T}$ that
\begin{equation}
K^*\left(  T\right)  = \frac{3\Vol\left(  T^{\#}\right)  ^{2}}{4\Vol\left(
T\right)  }. \label{detformula}
\end{equation}

For the last part of the theorem, we observe that there is a hexahedron determined by the vertices $v_0, c_{01}, c_{02}, c_{03}, c_{012}, c_{013}, c_{023}$, and $c_{0123}$ (see Figure \ref{fig:hexahedron}). 
\begin{figure}[t]
    \centering
    \includegraphics[width=.9\textwidth]{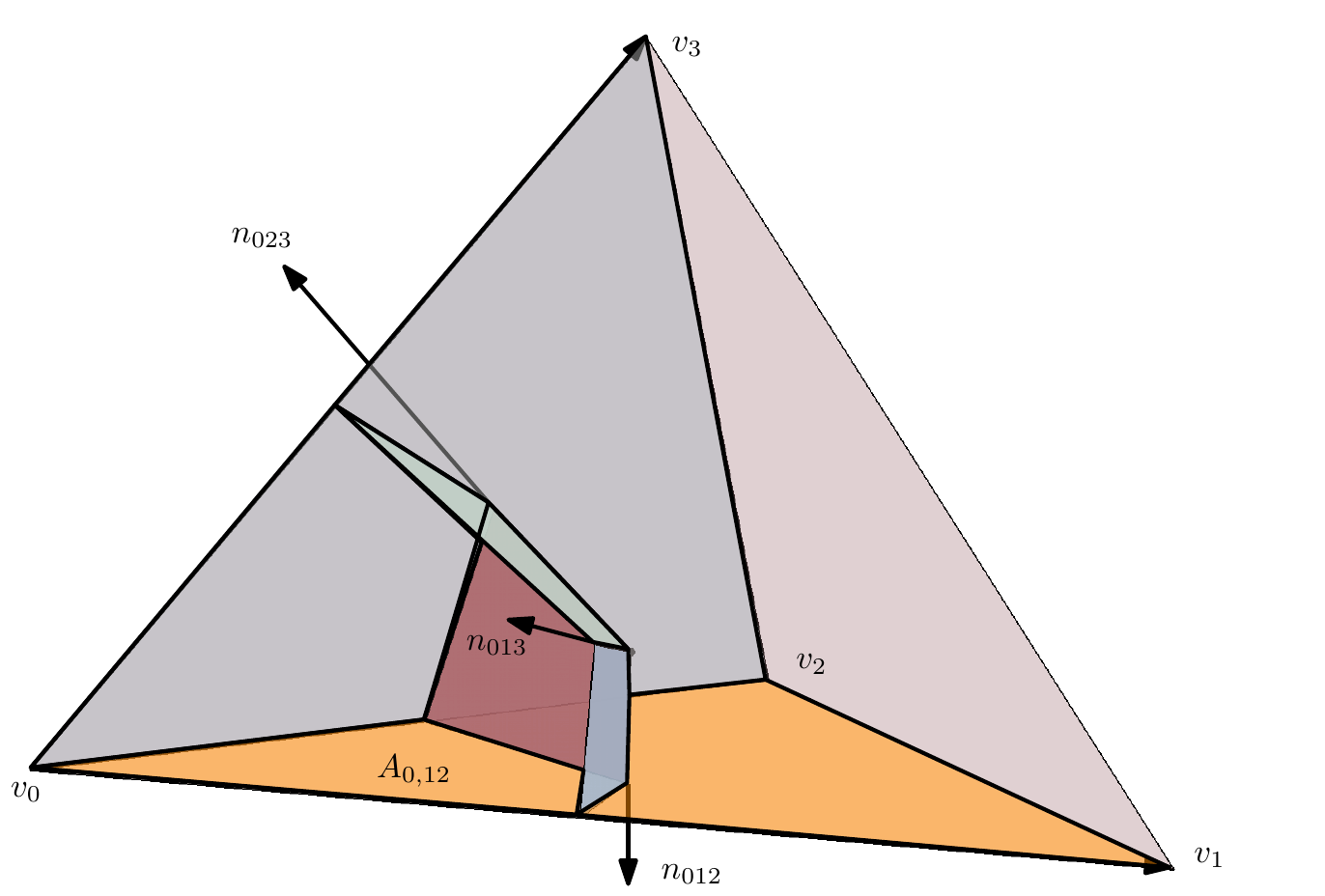}
    \caption{Around the bottom left vertex we see the hexahedron separating the piece of volume associated to vertex $0$ and some of the relevant normal vectors and areas.}
    \label{fig:hexahedron}
\end{figure}
Further, $n_0$, as defined, is equal to the 
sum of three face normals directed into this hexahedron, $n_{01}, n_{02}, n_{03}$. The other face normals are $\frac{A_{0,12}}{A_{012}}n_{012}$, $\frac{A_{0,13}}{A_{013}}n_{013}$, and $\frac{A_{0,23}}{A_{023}}n_{023}$, where $n_{ijk}$ are the face normals to the faces of the tetrahedron (with length equal to the area of the face) and $A_{i,jk}$ is the area from $[i,j,k]$ associated to vertex $i$ as described in Section \ref{sec:dualnd}. Since the hexahedron is a polyhedron, we have that 
\begin{equation}\label{eq:vertexnormalfromfacenorms}
n_0 = 
\frac{A_{0,12}}{A_{012}}n_{012} +\frac{A_{0,13}}{A_{013}}n_{013}+ \frac{A_{0,23}}{A_{023}}n_{023} 
.
\end{equation}
We have similar formulas for the other vertices.

It now follows that 
\[
\left[
\begin{array}[c]{cccc}%
0  & \frac{A_{0,23}}{A_{023}} & \frac{A_{0,13}}{A_{013}} & \frac{A_{0,12}}{A_{012}}\\
\frac{A_{1,23}}{A_{123}} & 0 & \frac{A_{1,03}}{A_{013}}  & \frac{A_{1,02}}{A_{012}}\\
\frac{A_{2,13}}{A_{123}}  & \frac{A_{2,03}}{A_{023}} & 0 & \frac{A_{2,01}}{A_{012}}\\
\frac{A_{3,12}}{A_{123}} & \frac{A_{3,02}}{A_{023}} & \frac{A_{3,01}}{A_{013}} & 0 %
\end{array}
\right]  \left[
\begin{array}
[c]{cc}%
n_{123} & A_{123}\\
n_{023} & A_{023}\\
n_{013} & A_{013}\\
n_{012} & A_{012}%
\end{array}
\right]  =\left[
\begin{array}
[c]{cc}%
n_{0} & A_{0}\\
n_{1} & A_{1}\\
n_{2} & A_{2}\\
n_{3} & A_{3}%
\end{array}
\right].
\]
Taking the determinant and using that the determinant is multilinear, we get that
\begin{align}
    -\frac{9 \Vol(T)^2 \det(A)}{2A_{123}A_{023}A_{013}A_{012}} =\det\left(\widehat{n^T}\right)=-6K^*(T) \Vol(T).
\label{eq:Knormal}
\end{align}

\subsection{Arbitrary Dimensions}
In this section, we give the argument for $N$ dimensions. We still have $-Lv^T=n^T$. Let $V_{\widehat{j}}$ denote $v^T$ with the $j$th row removed. Applying the Binet-Cauchy formula to $-\widehat{L}v^T=\widehat{n}^T$, we obtain

\[
\det\left(-\widehat{L}v^T\right) = \sum_{j = 0}^N \det\left(-\widehat{L}_{0j}\right) \det\left(V_{\widehat{j}}\right).
\]

When computing $\det\left(-\widehat{L}v^T\right)$, it is convenient to assume $v_0 = \vec{0}$ so all but the first term in this sum is $0$. This assumption can be made without loss of generality because if $C$ is the $(N+1)$ x $N$ matrix with all rows equal to $v_0$, then $\widehat{L}v^T = \widehat{L}\left(v^T - C\right)$, and so $\det\left(\widehat{L}v^T\right) = \det\left(\widehat{L}\left(v^T - C\right)\right)$. With this simplification, we obtain

\begin{align}
\det\left(-\widehat{L}v^T\right) = K^*\left(  T\right)\det\left(\widehat{v^T}\right). \label{eq:detvol}
\end{align}

But since $-\widehat{L}v^T=\widehat{n^T}$, the determinant on the left hand side is given by Proposition 2:
\[
\det\left(-\widehat{L}v^T\right) = \det\left(\widehat{n^T}\right) = \frac{(-1)^{N}N^{N}}{N!}\Vol(T^{\#})^{N-1}.
\]
Also, since $v_0 = \vec{0}$, we have that
\[
\det\left(\widehat{v^T}\right) = \det \left[ \begin{array}[c]{ccccc} 
    v_1-v_0 & v_2-v_0 & \cdots & v_N-v_0  \end{array} \right] = (-1)^N N! \Vol(T).
\]
Thus,
\begin{align}
    (-1)^N N! \Vol(T) K^*\left(  T\right) = \det\left(\widehat{n^T}\right) = \frac{(-1)^{N}N^{N}}{N!}\Vol(T^{\#})^{N-1}
\end{align}
or
\[
K^*\left(  T\right) = \frac{N^{N}\Vol(T^{\#})^{N-1}}{\left(N!\right)^2\Vol(T)}.
\]
Now (\ref{eq:KandA}) follows from an argument similar to the derivation of (\ref{eq:Knormal}). We note that 
\begin{align}
\left[
\begin{array}[c]{cccc}%
0  & \frac{A_{0,\widehat{1}}}{A_{\widehat{1}}} & \cdots & \frac{A_{0,\widehat{N}}}{A_{\widehat{N}}}\\
\frac{A_{1,\widehat{0}}}{A_{\widehat{0}}} & 0 & \cdots  & \frac{A_{1,\widehat{N}}}{A_{\widehat{N}}}\\
\vdots  & \ddots & \ddots & \vdots\\
\frac{A_{N,\widehat{0}}}{A_{\widehat{0}}} & \cdots & \frac{A_{N,\widehat{N-1}}}{A_{\widehat{N-1}}} & 0 %
\end{array}
\right]  \left[
\begin{array}
[c]{cc}%
n_{\widehat{0}} & A_{\widehat{0}}\\
n_{\widehat{1}} & A_{\widehat{1}}\\
\vdots & \vdots\\
n_{\widehat{N}} & A_{\widehat{N}}%
\end{array}
\right]  =\left[
\begin{array}
[c]{cc}%
n_{0} & A_{0}\\
n_{1} & A_{1}\\
\vdots & \vdots\\
n_{N} & A_{N}%
\end{array}
\right].
\end{align}
The result then follows from taking the determinant of both sides and using the multilinearity of the determinant, resulting in 
\begin{align*}
    \frac{\det(A)}{\prod_{i=0}^N A_{\widehat{i}}} \frac{ N^{N}}{N!}\Vol(T)^{N-1} =\det\left(\widehat{n^T}\right) = (-1)^N N! \Vol(T) K^*(T).
\end{align*}

\subsection{Remarks on \texorpdfstring{$T^\#$}{TEXT}}

In dimension $2$, we can easily see that $T^\#$ is the pedal triangle of the center of the original triangle. We will discuss this more in Section \ref{sec:simson}. As of yet, other than in dimension $2$ we do not know how the simplex $T^\#$ can be properly placed in the original tetrahedron $T$ or have any direct connections other than its definition. However, it is possible to construct $T^\#$ directly without resorting to Minkowski's Theorem in its generality. First compute its volume using (\ref{eq:volfromnormals}), then set the origin at a vertex, then use (\ref{eq:normalvertexinverse}) and invert the matrix on the left to get the matrix of vertices.

\section{Definiteness of Laplacian}

Proving conditions under which the Laplacian necessarily has nonzero determinant is of considerable interest. It is well-known that if the edge duals are positive then the Laplacian is an M-matrix and the definiteness follows directly. However, even in rudimentary cases such as the Laplacian for a non-Delaunay triangulation, it can be the case that the Laplacian is not an M-matrix but is still necessarily definite. In this section, we review what is known for the case of two dimensions and three-dimensional circle packings and describe the consequences of Theorem \ref{thm:determinant}.

\subsection{Triangles in two dimensions}\label{sec:simson}

The triangle case for (\ref{eq:keqVoverV}) was found in \cite{glickenstein3}. It can be seen that in dimension two, the triangle $T^\#$ is exactly the pedal triangle of the center point $C_{[i,j,k]}$. 
Much is known about the geometry of the pedal triangle, for instance \cite{johnsongeometry60,gallatly1910modern}. In particular, the triangle is degenerate (which is equivalent to having zero area) exactly when the center is on the circumcircle due to the Wallace-Simson Theorem (see \cite[pp. 137-138]{johnsongeometry60} for an interesting historical note). It is a direct consequence that the finite volume Laplacians arising from finite elements in two dimensions and from circle packings result have maximal rank since the centers are the circumcenter in the former and the incenter in the latter, both inside the circumcircle.

It is also known that if one takes an inversion of the triangle with respect to a circle centered at $C_{[i,j,k]}$ then the pedal triangle is similar to the inversion of the triangle itself. It follows that the pedal triangle is degenerate if the center $C_{[i,j,k]}$ lies on the circumcircle of the inversion of the triangle (and hence on the inversion of the circumcircle).

\subsection{Sphere packing in three dimensions}
Sphere packing considers tetrahedra arising from the centers of four mutually tangent spheres. Such a tetrahedron can equivalently be described as a tetrahedron with a sphere tangent to all of its six edges, sometimes called a circumscriptible tetrahedron \cite{court50}.
A version of Theorem \ref{thm:determinant} was proved for sphere packing in \cite[Proposition A1]{glickenstein2}, which shows that 
\begin{align*}
    K^*(T) &= \frac{36 r_i r_j r_k r_\ell V_{ijk\ell}}{P_{ijk}P_{ij\ell}P_{ik\ell}P_{jk\ell}}\\
\end{align*}
using \cite[Lemma 3]{glickenstein1}. Note that there is a slight error in the constant in \cite{glickenstein2}.



We recall from \cite{glickenstein1} that in the sphere packing case, 
\begin{align}
    A_{i,jk}= r_i r_{ijk} = \frac{2 r_i A_{ijk}}{P_{ijk}}
\label{eq:spherepackingarea}
\end{align}
where $r_{ijk}$ is the radius of the circle inscribed in triangle $ijk$ and $P_{ijk}$ is the perimeter of the triangle $ijk$. We can now use (\ref{eq:KandA}). First note that using the formula (\ref{eq:spherepackingarea}) and multilinearity of the determinant, in this case we have
\begin{align*}
    \det A &= \frac{16r_1r_2r_3r_4A_{123}A_{023}A_{013}A_{012}}{P_{123}P_{023}P_{013}P_{012}}  \det \left[\begin{array}
[c]{cccc}%
0 & 1 & 1 & 1\\
1 & 0 & 1 & 1\\
1 & 1 & 0 & 1\\
1 & 1 & 1 & 0
\end{array} \right]\\ 
&= -\frac{48r_1r_2r_3r_4A_{123}A_{023}A_{013}A_{012}}{P_{123}P_{023}P_{013}P_{012}}  .
\end{align*}
It then follows from Theorem \ref{thm:determinant} that 
\[
    K^*(T)=\frac{36r_1r_2r_3r_4\Vol(T)}{P_{123}P_{023}P_{013}P_{012}}.
\]

\subsection{General case}
In general, we can see the following: 

\begin{theorem}
Given a Euclidean $N$-simplex $\sigma$, the locus $L$ of points for which the center results in a degenerate simplex $T^{\#}$ will satisfy the following:
\begin{enumerate}
    \item The vertices of the simplex lie on $L$.
    \item The interior of the simplex does not intersect $L$.
    \item $L$ is described by the zeroes of a polynomial in $N$ variables of degree $N(N-1)$. 
\end{enumerate}
\end{theorem}

\begin{proof}
If the center is at vertex $v_i$ then the normals $n_{ij}$ would be zero for all $j$, implying the first statement. The second follows from the fact that the dual areas are positive for each point in the interior of the simplex, resulting in a standard weighted graph Laplacian on a connected graph, which cannot have $K^*(T) = 0$.

To derive the third, we simply note that given a point $x \in \R^N$, the projections to each of the sub-simplices are linear. Hence the normals $n_{ij}$ defined in (\ref{def:n}) are polynomial of degree $N-1$. The volume of the simplex $T^{\#}$ is zero if and only if $K^*(T)=0$, which is true if and only if the determinant of the matrix of normals is zero according to (\ref{eq:detvol}). This determinant is a polynomial of degree $N(N-1)$.
\end{proof}

\begin{figure}
    \centering
    \includegraphics[width=.8\textwidth]{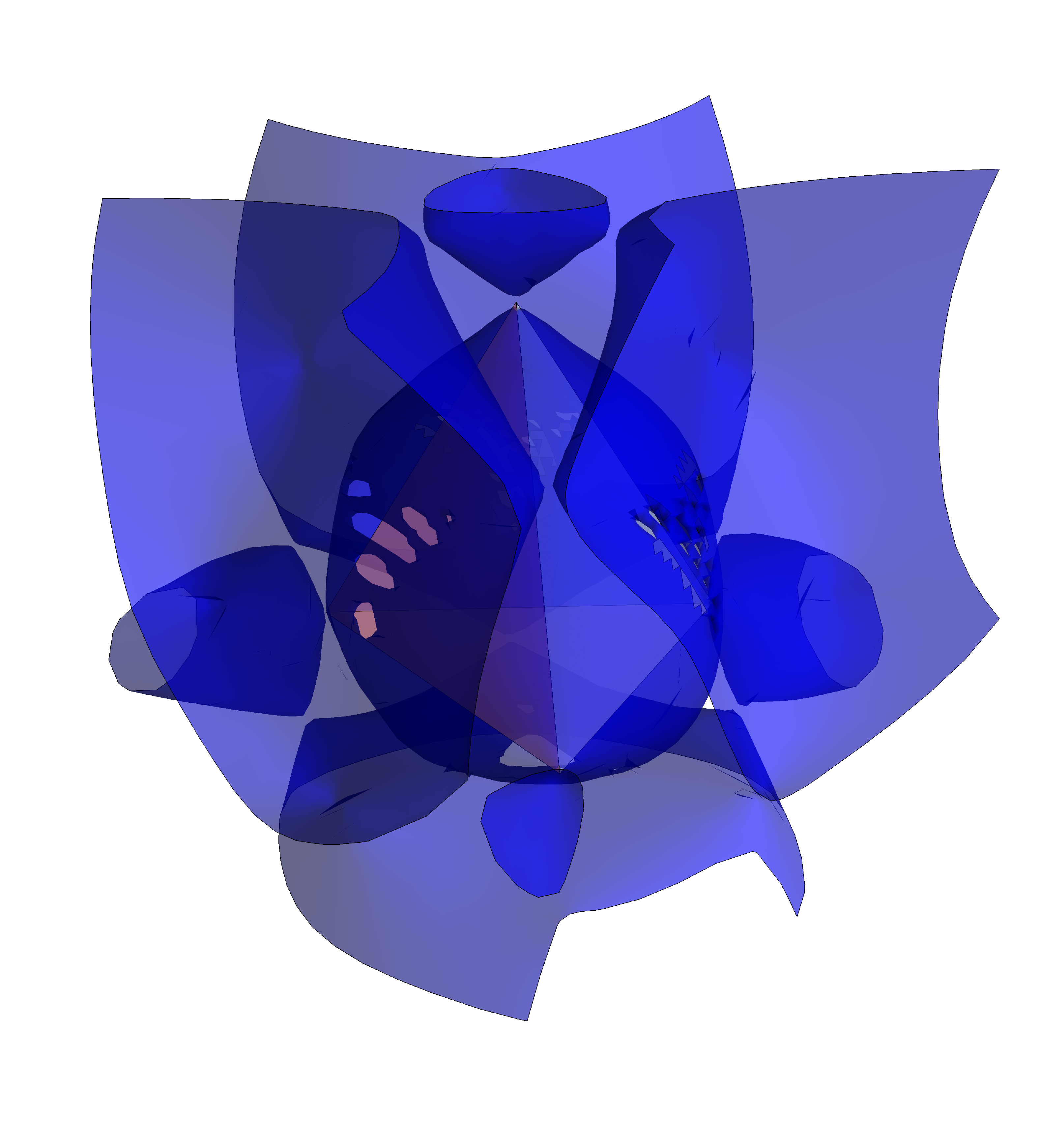}
    \caption{Locus of zeroes for the Laplacian for the tetrahedron shown.}
    \label{fig:sextic}
\end{figure}
An example of the equation for the locus for the regular tetrahedron in $\R^3$ is given by 
\begin{align*}
  &  \frac{x^6}{72}-\frac{x^5}{12}-\frac{23 x^4 y^2}{648}-\frac{55 x^4 y z}{162 \sqrt{2}}-\frac{x^4 y}{18 \sqrt{3}}\\
   & +\frac{4 x^4 z^2}{81}-\frac{x^4 z}{18
   \sqrt{6}}+\frac{5 x^4}{36}+\frac{23 x^3 y^2}{162}+\frac{55}{81} \sqrt{2} x^3 y z+\frac{2 x^3 y}{9 \sqrt{3}}-\frac{16 x^3 z^2}{81}\\
   &+\frac{1}{9}
   \sqrt{\frac{2}{3}} x^3 z+\frac{181 x^2 y^4}{1944}-\frac{55 x^2 y^3 z}{243 \sqrt{2}}-\frac{37 x^2 y^3}{81 \sqrt{3}}\\
   &+\frac{8}{81} x^2 y^2 z^2-\frac{x^2 y^2
   z}{9 \sqrt{6}}-\frac{23 x^2 y^2}{162}-\frac{10}{243} \sqrt{2} x^2 y z^3-\frac{10 x^2 y z^2}{27 \sqrt{3}}-\frac{55}{81} \sqrt{2} x^2 y z\\
   &+\frac{32 x^2
   z^4}{243}-\frac{22}{81} \sqrt{\frac{2}{3}} x^2 z^3+\frac{16 x^2 z^2}{81}-\frac{x^2}{9}-\frac{181 x y^4}{972}+\frac{55}{243} \sqrt{2} x y^3 z\\
   &+\frac{74 x
   y^3}{81 \sqrt{3}}-\frac{16}{81} x y^2 z^2+\frac{1}{9} \sqrt{\frac{2}{3}} x y^2 z+\frac{20}{243} \sqrt{2} x y z^3+\frac{20 x y z^2}{27 \sqrt{3}}\\
   &-\frac{4 x
   y}{9 \sqrt{3}}-\frac{64 x z^4}{243}+\frac{44}{81} \sqrt{\frac{2}{3}} x z^3-\frac{2}{9} \sqrt{\frac{2}{3}} x z+\frac{31 y^6}{5832}+\frac{55 y^5 z}{486
   \sqrt{2}}\\
   &+\frac{5 y^5}{486 \sqrt{3}}+\frac{4 y^4 z^2}{81}-\frac{467 y^4 z}{486 \sqrt{6}}-\frac{49 y^4}{972}+\frac{10}{729} \sqrt{2} y^3 z^3-\frac{34 y^3
   z^2}{243 \sqrt{3}}+\frac{67}{243} \sqrt{2} y^3 z\\
   &+\frac{32 y^2 z^4}{243}
   -\frac{86}{243} \sqrt{\frac{2}{3}} y^2 z^3-\frac{4 y^2
   z^2}{81}+\frac{y^2}{27}-\frac{64 y z^4}{243 \sqrt{3}}+\frac{44}{243} \sqrt{2} y z^3-\frac{2}{27} \sqrt{2} y z\\
   &-\frac{8 z^6}{729}+\frac{32}{243}
   \sqrt{\frac{2}{3}} z^5
   -\frac{40 z^4}{243}+\frac{2 z^2}{27}=0.
\end{align*}
A picture of the solution set for the regular tetrahedron is given in Figure \ref{fig:sextic}. In two dimensions,  of the locus is the zeroes of a quadratic polynomial in two variables that contains each of the triangle vertices. By looking at the symmetry of the polynomial, one can deduce that it is the circumcircle. We finally note that in the case of a tetrahedron, the locus is not the same as the locus of points such that the pedal tetrahedron is degenerate, which is given by a Cayley cubic surface \cite{court50} (also called the Steiner cubic surface \cite{steinercubic}). See also \cite{roanes, botana, roanes12}.



\

\bibliographystyle{alpha}
\bibliography{biblio}
\end{document}